\pdfoutput=1
\documentclass[12pt,letter]{amsart}
\usepackage{xcolor}
\usepackage[all,color]{xy}
\usepackage{amsmath,amssymb,amsthm,amscd,amsxtra,amsfonts,mathrsfs,graphicx,enumitem,bm,slashed,array,setspace,amsfonts,cite}
\input xy
\xyoption{all}
\newtheorem{Theorem}{Theorem}[section]
\newtheorem{Lemma}[Theorem]{Lemma}
\newtheorem{Proposition}[Theorem]{Proposition}
\newtheorem{Corollary}[Theorem]{Corollary}
\newtheorem{Example}[Theorem]{Example}
\newtheorem{Remark}[Theorem]{Remark}
\newtheorem{Definition}[Theorem]{Definition}

\newtheorem{Notation}[Theorem]{Notation}
\newtheorem{Notation/Lemma}[Theorem]{Notation/Lemma}

\newtheorem*{Theorem A}{Theorem A}

\newcommand*{\overbar}[1]{\mkern 1.5mu\overline{\mkern-1.5mu#1\mkern-1.5mu}\mkern 1.5mu}
\advance\evensidemargin-.5in
\advance\oddsidemargin-.5in
\advance\textwidth1in

\begin{document}
\author{Charlie Beil}
 \address{Institut f\"ur Mathematik und Wissenschaftliches Rechnen, Universit\"at Graz, Heinrichstrasse 36, 8010 Graz, Austria.}
 \email{charles.beil@uni-graz.at}
 \title[On the central geometry of nonnoetherian dimer algebras]{On the central geometry of\\ nonnoetherian dimer algebras}
 \date{}

\begin{abstract}
Let $Z$ be the center of a nonnoetherian dimer algebra on a torus.
Although $Z$ itself is also nonnoetherian, we show that it has Krull dimension $3$, and is locally noetherian on an open dense set of $\operatorname{Max}Z$.
Furthermore, we show that the reduced center $Z/\operatorname{nil}Z$ is depicted by a Gorenstein singularity, and contains precisely one closed point of positive geometric dimension.
\end{abstract}

\maketitle

\section{Introduction}

In this article, all dimer quivers are nondegenerate and embed in a two-torus. 
A dimer algebra $A$ is noetherian if and only if its center $Z$ is noetherian, if and only if $A$ is a noncommutative crepant resolution of its 3-dimensional toric Gorenstein center \cite{B4, Br, D}.
We show that the center $Z$ of a nonnoetherian dimer algebra is also $3$-dimensional, and may be viewed as the coordinate ring for a toric Gorenstein singularity that has precisely one `smeared-out' point of positive geometric dimension.
This is made precise using the notion of a depiction, which is a finitely generated overring that is as close as possible to $Z$, in a suitable geometric sense (Definition \ref{depiction def}).

Denote by $\operatorname{nil}Z$ the nilradical of $Z$, and by $\hat{Z} := Z/\operatorname{nil}Z$ the reduced ring of $Z$.
Our main theorem is the following.

\begin{Theorem} \label{big theorem2} (Theorems \ref{generically noetherian}, \ref{hopefully...}.)
Let $A$ be a nonnoetherian dimer algebra with center $Z$, and let $\Lambda := A/\left\langle p - q \ | \ \text{$p,q$ a non-cancellative pair} \right\rangle$ be its ghor algebra with center $R$.
 \begin{enumerate}
  \item The nonnoetherian rings $Z$, $\hat{Z}$, and $R$ each have Krull dimension 3, and the integral domains $\hat{Z}$ and $R$ are depicted by the cycle algebras of $A$ and $\Lambda$.
  \item The reduced scheme of $\operatorname{Spec}Z$ and the scheme $\operatorname{Spec}R$ are birational to a noetherian affine scheme, and each contain precisely one closed point of positive geometric dimension.
\end{enumerate}
\end{Theorem}

Dimer models were introduced in string theory in 2005 in the context of brane tilings \cite{HK, FHVWK}.
The dimer algebra description of the combinatorial data of a brane tiling arose from the notion of a superpotential algebra (or quiver with potential), which was introduced a few years earlier in \cite{BD}.
Stable (i.e., `superconformal') brane tilings quickly made their way to the mathematics side, but the more difficult study of unstable brane tilings was largely left open, in regards to both their mathematical and physical properties. 
There were two main difficulties: in contrast to the stable case, the `mesonic chiral ring' (closely related to what we call the cycle algebra\footnote{The mesonic chiral ring is the ring of gauge invariant operators; these are elements of the dimer algebra that are invariant under isomorphic representations, and thus are cycles in the quiver.}) did not coincide with the center of the dimer algebra, and (ii) although the mesonic chiral ring still appeared to be a nice ring, the center certainly was not.

The center is supposed to be the coordinate ring for an affine patch on the extra six compact dimensions of spacetime, the so-called (classical) vacuum geometry.
But examples quickly showed that the center of an unstable brane tiling could be infinitely generated.
To say that the vacuum geometry was a nonnoetherian scheme -- something believed to have no visual representation  or concrete geometric interpretation -- was not quite satisfactory from a physics perspective. 
However, unstable brane tilings are physically allowable theories.
To make matters worse, almost all brane tilings are unstable, and it is only in the case of a certain uniform symmetry (an `isoradial embedding') that they become stable.
Moreover, in the context of 11-dimensional M-theory, stable and unstable brane tilings are equally `good'.
The question thus remained:
\begin{center}
\textit{What does the vacuum geometry of an unstable brane tiling look like?}
\end{center}
The aim of this article is to provide an answer.
In short, the vacuum geometry of an unstable brane tiling looks just like the vacuum geometry of a stable brane tiling, namely a $3$-dimensional complex cone, except that there is precisely one curve or surface passing through the apex of the cone that is identified as a single `smeared-out' point.

\section{Preliminary definitions}

Throughout, $k$ is an uncountable algebraically closed field.
Given a quiver $Q$, we denote by $kQ$ the path algebra of $Q$; and by $Q_0$ and $Q_1$ the sets of vertices and arrows of $Q$ respectively. 
The vertex idempotent at vertex $i \in Q_0$ is denoted $e_i$, and the head and tail maps are denoted $\operatorname{h},\operatorname{t}: Q_1 \to Q_0$.
By monomial, we mean a nonconstant monomial.

\subsection{Dimer algebras, ghor algebras, and cyclic contractions}

\begin{Definition} \label{dimer def} \rm{ \

$\bullet$ Let $Q$ be a finite quiver whose underlying graph $\overbar{Q}$ embeds into a real two-torus $T^2$, such that each connected component of $T^2 \setminus \overbar{Q}$ is simply connected and bounded by an oriented cycle of length at least $2$, called a \textit{unit cycle}.\footnote{In forthcoming work, we consider the nonnoetherian central geometry of ghor algebras on higher genus surfaces.  Dimer quivers on other surfaces arise in contexts such as Belyi maps, cluster categories, and bipartite field theories; see e.g., \cite{BGH, BKM, FGU}.} 
The \textit{dimer algebra} of $Q$ is the quiver algebra $A := kQ/I$ with relations
$$I := \left\langle p - q \ | \ \exists \ a \in Q_1 \text{ such that } pa \text{ and } qa \text{ are unit cycles} \right\rangle \subset kQ,$$
where $p$ and $q$ are paths.

Since $I$ is generated by certain differences of paths, we may refer to a path modulo $I$ as a \textit{path} in the dimer algebra $A$.

$\bullet$ Two paths $p,q \in A$ form a \textit{non-cancellative pair} if $p \not = q$, and there is a path $r \in kQ/I$ such that
$$rp = rq \not = 0 \ \ \text{ or } \ \ pr = qr \not = 0.$$
$A$ and $Q$ are called \textit{non-cancellative} if there is a non-cancellative pair; otherwise they are called \textit{cancellative}.

$\bullet$ The \textit{ghor algebra} of $Q$ is the quotient
$$\Lambda := A/\left\langle p - q \ | \ p,q \text{ is a non-cancellative pair} \right\rangle.$$ 
A dimer algebra $A$ coincides with its ghor algebra if and only if $A$ is cancellative, if and only if $A$ is noetherian, if and only if $\Lambda$ is noetherian \cite[Theorem 1.1]{B4}.

$\bullet$ Let $A$ be a dimer algebra with quiver $Q$.
 \begin{itemize}
  \item[--] A \textit{perfect matching} $D \subset Q_1$ is a set of arrows such that each unit cycle contains precisely one arrow in $D$.
  \item[--] A \textit{simple matching} $D \subset Q_1$ is a perfect matching such that $Q \setminus D$ supports a simple $A$-module of dimension $1^{Q_0}$ (that is, $Q \setminus D$ contains a cycle that passes through each vertex of $Q$).
 Denote by $\mathcal{S}$ the set of simple matchings of $A$.
  \item[--] $A$ is said to be \textit{nondegenerate} if each arrow of $Q$ belongs to a perfect matching.\footnote{For our purposes, it suffices to assume that each cycle contains an arrow that belongs to a perfect matching; see \cite{B1}.}
 \end{itemize}

Each perfect matching $D$ defines a map
$$n_D: Q_{\geq 0} \to \mathbb{Z}_{\geq 0}$$
that sends path $p$ to the number of arrow subpaths of $p$ that are contained in $D$.
$n_D$ is additive on concatenated paths, and if $p,p' \in Q_{\geq 0}$ are paths satisfying $p + I = p' + I$, then $n_D(p) = n_D(p')$.
In particular, $n_D$ induces a well-defined map on the paths of $A$.
}\end{Definition}

Now consider dimer algebras $A = kQ/I$ and $A' = kQ'/I'$, and suppose $Q'$ is obtained from $Q$ by contracting a set of arrows $Q_1^* \subset Q_1$ to vertices.
This contraction defines a $k$-linear map of path algebras
$$\psi: kQ \to kQ'.$$
If $\psi(I) \subseteq I'$, then $\psi$ induces a $k$-linear map of dimer algebras, called a \textit{contraction},
$$\psi: A \to A'.$$

Denote by
$$B := k\left[x_D : D \in \mathcal{S}' \right]$$
the polynomial ring generated by the simple matchings $\mathcal{S}'$ of $A'$.
To each path $p \in A'$, associate the monomial
$$\bar{\tau}(p) := \prod_{D \in \mathcal{S}'} x_D^{n_D(p)} \in B.$$
For each $i,j \in Q'_0$, this association extends to a $k$-linear map 
$$\bar{\tau}: e_jA'e_i \to B.$$
This map is an algebra homomorphism if $i = j$, and injective if $A'$ is cancellative \cite[Proposition 4.29]{B2}.
For each $i,j \in Q_0$, we also consider the $k$-linear map given by the composition of the $k$-linear maps $\psi$ and $\bar{\tau}$,
$$\bar{\tau}_{\psi}: e_jAe_i \stackrel{\psi}{\longrightarrow} e_{\psi(j)}A'e_{\psi(i)} \stackrel{\bar{\tau}}{\longrightarrow} B.$$ 
Given $p \in e_jAe_i$ and $q \in e_{\ell}A'e_k$, we will write
$$\overbar{p} := \bar{\tau}_{\psi}(p) := \bar{\tau}(\psi(p)) \ \ \ \text{ and } \ \ \ \overbar{q} := \bar{\tau}(q).$$

$\psi$ is called a \textit{cyclic contraction} if $A'$ is cancellative and
\begin{equation*} \label{cycle algebra}
S := k \left[ \cup_{i \in Q_0} \bar{\tau}_{\psi}(e_iAe_i) \right] = k \left[ \cup_{i \in Q'_0} \bar{\tau}(e_iA'e_i) \right] =: S'.
\end{equation*}
In this case, we call $S$ the \textit{cycle algebra} of $A$.
The cycle algebra is independent of the choice of cyclic contraction $\psi$ \cite[Theorem 3.14]{B3}, and is isomorphic to the center of $A'$ \cite[Theorem 1.1.3]{B2}.
Moreover, every nondegenerate dimer algebra admits a cyclic contraction \cite[Theorem 1.1]{B1}.

In addition to the cycle algebra, the \textit{ghor center} of $A$,
$$R := k\left[ \cap_{i \in Q_0} \bar{\tau}_{\psi}(e_iAe_i) \right],$$
also plays an important role. 
This algebra is isomorphic to the center of the ghor algebra $\Lambda$ of $Q$ \cite[Theorem 1.1.3]{B2}.

\begin{Notation} \rm{
Let $\pi: \mathbb{R}^2 \rightarrow T^2$ be a covering map such that for some $i \in Q_0$,
$$\pi\left(\mathbb{Z}^2 \right) = i.$$
Denote by $Q^+ := \pi^{-1}(Q) \subset \mathbb{R}^2$ the covering quiver of $Q$.
For each path $p$ in $Q$, denote by $p^+$ the unique path in $Q^+$ with tail in the unit square $[0,1) \times [0,1) \subset \mathbb{R}^2$ satisfying $\pi(p^+) = p$. 
For $u \in \mathbb{Z}^2$, denote by $\mathcal{C}^u$ the set of cycles $p$ in $A$ such that
$$\operatorname{h}(p^+) = \operatorname{t}(p^+) + u \in Q_0^+.$$
}\end{Notation}

\begin{Notation/Lemma} \label{sigmalemma} \rm{
We denote by $\sigma_i \in A$ the unique unit cycle (modulo $I$) at $i \in Q_0$, and by $\sigma$ the monomial
$$\sigma := \overbar{\sigma}_i = \prod_{D \in \mathcal{S}'}x_D.$$
The sum $\sum_{i \in Q_0} \sigma_i$ is a central element of $A$.
}\end{Notation/Lemma}

\begin{Lemma} \label{cyclelemma} \
\begin{enumerate}
 \item If $p \in \mathcal{C}^{(0,0)}$ is nontrivial, then
$\overbar{p} = \sigma^n$ for some $n \geq 1$.
 \item Let $u \in \mathbb{Z}^2$ and $p,q \in \mathcal{C}^u$.
Then $\overbar{p} = \overbar{q} \sigma^n$ for some $n \in \mathbb{Z}$. 
\end{enumerate}
\end{Lemma}

\begin{proof}
(1) is \cite[Lemma 5.2.1]{B2}, and (2) is \cite[Lemma 4.18]{B2}.
\end{proof}

\subsection{Nonnoetherian geometry: depictions and geometric dimension}

Let $S$ be an integral domain and a finitely generated $k$-algebra, and let $R$ be a (possibly nonnoetherian) subalgebra of $S$.
Denote by $\operatorname{Max}S$, $\operatorname{Spec}S$, and $\operatorname{dim}S$ the maximal spectrum (or variety), prime spectrum (or affine scheme), and Krull dimension of $S$ respectively; similarly for $R$.
For a subset $I \subset S$, set $\mathcal{Z}_S(I) := \left\{ \mathfrak{n} \in \operatorname{Max}S \ | \ \mathfrak{n} \supseteq I \right\}$.

\begin{Definition} \label{depiction def} \rm{\cite[Definition 3.1]{B5}
\begin{itemize}
 \item We say $S$ is a \textit{depiction} of $R$ if the morphism\footnote{The morphism $\iota_{S/R}$ is well-defined: Suppose $\mathfrak{q} \in \operatorname{Spec}S$.  
Then $S/\mathfrak{q}$ is an integral domain.
Whence the subalgebra $R/(\mathfrak{q} \cap R) \subseteq S/\mathfrak{q}$ is an integral domain.
Thus $\mathfrak{q} \cap R$ is a prime ideal of $R$.}
$$\iota_{S/R}: \operatorname{Spec}S \rightarrow \operatorname{Spec}R, \ \ \ \ \mathfrak{q} \mapsto \mathfrak{q} \cap R,$$
is surjective, and
\begin{equation} \label{condition}
\left\{ \mathfrak{n} \in \operatorname{Max}S \ | \ R_{\mathfrak{n}\cap R} = S_{\mathfrak{n}} \right\} = \left\{ \mathfrak{n} \in \operatorname{Max}S \ | \ R_{\mathfrak{n} \cap R} \text{ is noetherian} \right\} \not = \emptyset.
\end{equation}
 \item The \textit{geometric height} of $\mathfrak{p} \in \operatorname{Spec}R$ is the minimum
$$\operatorname{ght}(\mathfrak{p}) := \operatorname{min} \left\{ \operatorname{ht}_S(\mathfrak{q}) \ | \ \mathfrak{q} \in \iota^{-1}_{S/R}(\mathfrak{p}), \ S \text{ a depiction of } R \right\}.$$
The \textit{geometric dimension} of $\mathfrak{p}$ is
$$\operatorname{gdim} \mathfrak{p} := \operatorname{dim}R - \operatorname{ght}(\mathfrak{p}).$$
 \end{itemize}
} \end{Definition}

We will denote the subsets (\ref{condition}) of the algebraic variety $\operatorname{Max}S$ by
\begin{align} \label{U U*}
\begin{split}
U_{S/R} & := \left\{ \mathfrak{n} \in \operatorname{Max}S \ | \ R_{\mathfrak{n}\cap R} = S_{\mathfrak{n}} \right\},\\ U^*_{S/R} & := \left\{ \mathfrak{n} \in \operatorname{Max}S \ | \ R_{\mathfrak{n} \cap R} \text{ is noetherian} \right\}.
\end{split}
\end{align}
The subset $U_{S/R}$ is open in $\operatorname{Max}S$ \cite[Proposition 2.4.2]{B5}.

\begin{Example} \rm{
Let $S = k[x,y]$, and consider its nonnoetherian subalgebra
$$R = k[x,xy,xy^2, \ldots] = k + xS.$$
$R$ is then depicted by $S$, and the closed point $xS \in \operatorname{Max}R$ has geometric dimension 1 \cite[Proposition 2.8]{B5}.
Furthermore, $U_{S/R}$ is the complement of the line
$$\mathcal{Z}(x) = \left\{ x = 0 \right\} \subset \operatorname{Max}S.$$
In particular, $\operatorname{Max}R$ may be viewed as 2-dimensional affine space $\mathbb{A}_k^2 = \operatorname{Max}S$ with the line $\mathcal{Z}(x)$ identified as a single `smeared-out' point.
From this perspective, $xS$ is a positive dimensional point of $\operatorname{Max}R$.
}\end{Example}

In the next section, we will show that the reduced center and ghor center of a nonnoetherian dimer algebra are both depicted by its cycle algebra, and both contain precisely one point of positive geometric dimension. 

\section{Proof of main theorem}

Throughout, $A$ is a nonnoetherian dimer algebra with center $Z$ and reduced center $\hat{Z} := Z/\operatorname{nil}Z$.
By assumption $A$ is nondegenerate, and thus there is a cyclic contraction $\psi: A \to A'$ to a noetherian dimer algebra $A'$ \cite[Theorem 1.1]{B1}.

The center $Z'$ of $A'$ is isomorphic to the cycle algebra $S$ \cite[Theorem 1.1.3]{B2}, and the reduced center $\hat{Z}$ of $A$ is isomorphic to a subalgebra of $R$ \cite[Theorem 4.1]{B6}.\footnote{It is often the case that $\hat{Z}$ is isomorphic to $R$; an example where $\hat{Z} \not \cong R$ is given in \cite[Example 4.3]{B6}.} 
We may therefore write
\begin{equation} \label{ZKA}
\hat{Z} \subseteq R.
\end{equation}

The following structural results will be useful.

\begin{Lemma} \label{cyclelemma2}
Let $g \in B$ be a monomial.
\begin{enumerate}[resume]
 \item If $g \in R$, $h \in S$ are monomials and $g \not = \sigma^n$ for each $n \geq 0$, then $gh \in R$.
 \item If $g \in R$ and $\sigma \nmid g$, then $g \in \hat{Z}$.
 \item If $g \in R$, then there is some $m \geq 1$ such that $g^m \in \hat{Z}$.
 \item If $g \in S$, then there is some $m \geq 0$ such that for each $n \geq 1$, $g^n \sigma^m \in \hat{Z}$.
 \item If $g \sigma \in S$, then $g \in S$.
 \item If a monomial $h \in S \setminus R$ satisfies $\sigma \nmid h$, then $h^n \not \in R$ for each $n \geq 1$.  Furthermore, a monomial $h \in S \setminus R$ exists for which $\sigma \nmid h$.
\end{enumerate}
\end{Lemma}

\begin{proof}
 (1) is \cite[Lemma 6.1]{B6}; (2) - (4) is \cite[Lemma 5.3]{B6}; and (5) is \cite[Lemma 4.18]{B2} for $u \not = (0,0)$, and Lemma \ref{cyclelemma}.1 for $u = (0,0)$; and (6) is \cite[Proposition 3.14]{B4}.
\end{proof}

\begin{Lemma} \label{S noetherian}
The cycle algebra $S$ is a finite type integral domain. 
\end{Lemma}

\begin{proof}
The cycle algebra $S$ is generated by the $\bar{\tau}_{\psi}$-images of cycles in $Q$ with no nontrivial cyclic proper subpaths.
Since $Q$ is finite, there is only a finite number of such cycles.
Therefore $S$ is a finitely generated $k$-algebra.
$S$ is also an integral domain since it is a subalgebra of the polynomial ring $B$. 
\end{proof}

It is well-known that the Krull dimension of the center of any cancellative dimer algebra (on a torus) is $3$ (e.g.\ \cite{Br}). 
The isomorphism $S \cong Z'$ therefore implies that the Krull dimension of the cycle algebra $S$ is $3$. 
In the following we give a new and independent proof of this result.

\begin{Lemma} \label{Jess}
The cycle algebra $S$ has Krull dimension $3$.
\end{Lemma}

\begin{proof}
Fix $j \in Q'_0$ and cycles in $e_jA'e_j$,
\begin{equation} \label{uv}
s_1 \in \mathcal{C}'^{(1,0)}, \ \ \ \ t_1 \in \mathcal{C}'^{(-1,0)}, \ \ \ \ s_2 \in \mathcal{C}'^{(0,1)}, \ \ \ \ t_2 \in \mathcal{C}'^{(0,-1)}.
\end{equation}
Consider the algebra
$$T := k[\sigma, \overbar{s}_1,\overbar{s}_2,\overbar{t}_1,\overbar{t}_2 ] \subseteq S' \stackrel{\textsc{(i)}}{=} S,$$
where (\textsc{i}) holds since the contraction $\psi$ is cyclic. 

Since $A'$ is cancellative, if 
$$p \in \mathcal{C}'^u \ \ \ \text{ and } \ \ \ q \in \mathcal{C}'^v$$ 
are cycles in $A'$ satisfying $\overbar{p} = \overbar{q}$, then $u = v$ \cite[Lemma 3.9]{B4}.
Thus there are no relations among the monomials $\overbar{s}_1,\overbar{s}_2,\overbar{t}_1,\overbar{t}_2$, by our choice of cycles (\ref{uv}).
However, by Lemma \ref{cyclelemma}.1, there are integers $n_1,n_2 \geq 1$ such that
\begin{equation} \label{siti}
\overbar{s}_1 \overbar{t}_1 = \sigma^{n_1} \ \ \ \text{ and } \ \ \ \overbar{s}_2 \overbar{t}_2 = \sigma^{n_2}.
\end{equation}

(i) We claim that $\operatorname{dim}T = 3$.
Since $T$ is a finite type integral domain and $k$ is algebraically closed, the variety $\operatorname{Max}T$ is equidimensional \cite[Ch.\ 13, Theorem A]{E}.
It thus suffices to show that the chain of ideals of $T$,
\begin{equation} \label{hope}
0 \subset (\sigma, \overbar{s}_1,\overbar{s}_2) \subseteq (\sigma, \overbar{s}_1, \overbar{s}_2, \overbar{t}_1) \subseteq (\sigma, \overbar{s}_1,\overbar{s}_2,\overbar{t}_1,\overbar{t}_2 ),
\end{equation}
is a maximal chain of distinct primes. 

The inclusions in (\ref{hope}) are strict since the relations among the monomial generators are generated by the two relations (\ref{siti}).

Moreover, (\ref{hope}) is a maximal chain of primes of $T$:
Suppose $\overbar{s}_i$ is in a prime $\mathfrak{p}$ of $T$. 
Then $\sigma$ is in $\mathfrak{p}$, by (\ref{siti}).
Whence $\overbar{s}_{i+1}$ or $\overbar{t}_{i+1}$ is also in $\mathfrak{p}$, again by (\ref{siti}).
Thus $(\sigma, \overbar{s}_1,\overbar{s}_2)$ is a minimal prime of $T$.

(ii) We now claim that $\operatorname{dim}S = \operatorname{dim}T$.
By Lemma \ref{cyclelemma}.2, we have 
$$S[\sigma^{-1}] = T[\sigma^{-1}].$$
Furthermore, $S$ and $T$ are finite type integral domains, by Lemma \ref{S noetherian}. 
Thus $\operatorname{Max}S$ and $\operatorname{Max}T$ are irreducible algebraic varieties that are isomorphic on their open dense sets $\{\sigma \not = 0\}$. 
Therefore $\operatorname{dim}S = \operatorname{dim}T$.
\end{proof}

\begin{Corollary}
The Krull dimension of the center of a noetherian dimer algebra is $3$.
\end{Corollary}

\begin{proof}
Follows from Lemma \ref{Jess} and the isomorphism $Z' \cong S$.
\end{proof}

\begin{Lemma} \label{max ideal}
The morphisms
\begin{equation} \label{max surjective}
\begin{array}{rcl}
\kappa_{S/\hat{Z}}: \operatorname{Max}S \to \operatorname{Max}\hat{Z}, & \ & \mathfrak{n} \mapsto \mathfrak{n} \cap \hat{Z},\\
\kappa_{S/R}: \operatorname{Max}S \to \operatorname{Max}R, & & \mathfrak{n} \mapsto \mathfrak{n} \cap R,
\end{array}
\end{equation}
and
$$\begin{array}{rcl}
\iota_{S/\hat{Z}}: \operatorname{Spec}S \to \operatorname{Spec}\hat{Z}, & \  & \mathfrak{q} \mapsto \mathfrak{q} \cap \hat{Z},\\
\iota_{S/R}: \operatorname{Spec}S \to \operatorname{Spec}R, & & \mathfrak{q} \mapsto \mathfrak{q} \cap R,
\end{array}$$
are well-defined and surjective.
\end{Lemma}

\begin{proof}
(i) We first claim that $\kappa_{S/\hat{Z}}$ and $\kappa_{S/R}$ are well-defined maps.
Indeed, let $\mathfrak{n}$ be in $\operatorname{Max}S$.
By Lemma \ref{S noetherian}, $S$ is of finite type, and by assumption $k$ is algebraically closed.
Therefore the intersections $\mathfrak{n} \cap \hat{Z}$ and $\mathfrak{n} \cap R$ are maximal ideals of $\hat{Z}$ and $R$ respectively (e.g., \cite[Lemma 2.1]{B5}).

(ii) We claim that $\kappa_{S/\hat{Z}}$ and $\kappa_{S/R}$ are surjective.
Fix $\mathfrak{m} \in \operatorname{max}\hat{Z}$.
Then $S\mathfrak{m}$ is a proper ideal of $S$ since $S$ is a subalgebra of the polynomial ring $B$.
Thus, since $S$ is noetherian, there is a maximal ideal $\mathfrak{n} \in \operatorname{Max}S$ containing $S\mathfrak{m}$.
Whence,
$$\mathfrak{m} \subseteq S\mathfrak{m} \cap \hat{Z} \subseteq \mathfrak{n} \cap \hat{Z}.$$
But $\mathfrak{n} \cap \hat{Z}$ is a maximal ideal of $\hat{Z}$ by Claim (i).
Therefore $\mathfrak{m} = \mathfrak{n} \cap \hat{Z}$.
Similarly, $\kappa_{S/R}$ is surjective.

(iii) It is clear that $\iota_{S/\hat{Z}}$ and $\iota_{S/R}$ are well-defined maps (see footnote 4).
Finally, we claim that $\iota_{S/\hat{Z}}$ and $\iota_{S/R}$ are surjective.
By \cite[Lemma 3.6]{B5}, if $D$ is a finitely generated algebra over an uncountable field $k$, and $C \subseteq D$ is a subalgebra, then $\iota_{D/C}: \operatorname{Spec}D \to \operatorname{Spec}C$ is surjective if and only if $\kappa_{D/C}:\operatorname{Max}D \to \operatorname{Max}C$ is surjective.
Therefore, $\iota_{S/\hat{Z}}$ and $\iota_{S/R}$ are surjective by Claim (ii).
\end{proof}

\begin{Lemma} \label{sigma in m}
If $\mathfrak{p} \in \operatorname{Spec}\hat{Z}$ contains a monomial, then $\mathfrak{p}$ contains $\sigma$.
\end{Lemma}

\begin{proof}
Suppose $\mathfrak{p}$ contains a monomial $g$.
Then there is a nontrivial cycle $p$ such that $\overbar{p} = g$.

Let $q^+$ be a path from $\operatorname{h}(p^+)$ to $\operatorname{t}(p^+)$.
Then the concatenated path $(pq)^+$ is a cycle in $Q^+$.
Thus, there is some $n \geq 1$ such that $\overbar{p}\overbar{q} = \sigma^n$, by Lemma \ref{cyclelemma}.1.

By Lemma \ref{max ideal}, there is a prime ideal $\mathfrak{q} \in \operatorname{Max}S$ such that $\mathfrak{q} \cap R = \mathfrak{p}$.
Furthermore, $\overbar{p}  \overbar{q} = \sigma^n$ is in $\mathfrak{q}$ since $\overbar{p} \in \mathfrak{p}$ and $\overbar{q} \in S$.
Whence $\sigma$ is also in $\mathfrak{q}$ since $\mathfrak{q}$ is prime.
But $\sigma \in \hat{Z}$. 
Therefore $\sigma \in \mathfrak{q} \cap \hat{Z} = \mathfrak{p}$.
\end{proof}

Denote the origin of $\operatorname{Max}S$ by
$$\mathfrak{n}_0  := \left( \overbar{s} \in S \ | \ s \text{ a nontrivial cycle} \right)S \in \operatorname{Max}S.$$
Consider the maximal ideals of $\hat{Z}$ and $R$ respectively,
$$\mathfrak{z}_0 := \mathfrak{n}_0 \cap \hat{Z} \ \ \text{ and } \ \ \mathfrak{m}_0 := \mathfrak{n}_0 \cap R.$$

\begin{Proposition} \label{local nonnoetherian}
The localizations $\hat{Z}_{\mathfrak{z}_0}$ and $R_{\mathfrak{m}_0}$ are nonnoetherian.
\end{Proposition}

\begin{proof}
There is a monomial $g \in S \setminus R$ such that $g^n \not \in R$ for each $n \geq 1$, by Lemma \ref{cyclelemma2}.6.

(i) Fix $n \geq 1$.  
We claim that $g^n$ is not in the localization $R_{\mathfrak{m}_0}$.
Assume otherwise; then there is an $a \in R$, $b \in \mathfrak{m}_0$, $\beta \in k^{\times}$, such that
$$g^n = \frac{a}{\beta + b} \in R_{\mathfrak{m}_0},$$
since $\mathfrak{m}_0$ is a maximal ideal of $R$.
Whence,
\begin{equation} \label{keeptry}
b g^n - a = -\beta g^n \not \in R.
\end{equation}
Thus $b g^n \not \in R$ since $a \in R$ and $\beta g^n \not \in R$.

Since $\mathfrak{m}_0$ is generated by monomials in $R$, we may write the polynomial $b \in \mathfrak{m}_0$ as a sum of monomials in $\mathfrak{m}_0$,
$$b = \sum_j b_j + \sum_k b'_k,$$
where for each $j$ and $k$, 
$$b_j g^n \not \in R \ \ \ \text{ and } \ \ \ b'_k g^n \in R.$$
Consider the polynomial
$$h := \sum_j b_jg^n + \beta g^n = a - \sum_k b'_k g^n.$$

If $h = 0$, then $\sum_j b_j g^n = - \beta g^n$.
Whence $\sum_j b_j = -\beta$ since $B$ is an integral domain.
But $\sum_j b_j$ is in $\mathfrak{m}_0$ whereas $\beta \in k^{\times}$ is not, a contradiction.
Therefore $h \not = 0$.

View $R \subset S$ as $k$-vector subspaces of the polynomial ring $B$.
These subspaces have bases given by all monomials with scalar coefficient $1$ in $R$ and $S$ respectively. 
Define a symmetric bilinear form on $S$: for monomials $m_1, m_2 \in S$ in the basis, set
$$(m_1,m_2) := \left\{ \begin{array}{cl} 1 & \text{ if } m_1 = m_2\\ 0 & \text{ otherwise } \end{array} \right.,$$
and extend $k$-bilinearly to $S$.

Now $h$ is nonzero and equals $\sum_j b_j g^n + \beta g^n$, and thus may be written entirely in terms of basis vectors that lie in the orthogonal complement to $R$ (with respect to $(\cdot, \cdot)$).
Thus $h$ itself does not lie in $R$.
But $h$ also equals $a - \sum_k b'_k g^n$, and thus lies in $R$, a contradiction.
Therefore $g^n$ is not in the localization $R_{\mathfrak{m}_0}$. 

(ii) We now claim that $R_{\mathfrak{m}_0}$ is nonnoetherian.

By Lemma \ref{cyclelemma2}.4, there is an $m \geq 1$ such that for each $n \geq 1$, 
$$h_n := g^n\sigma^m \in R.$$
Consider the chain of ideals of $R_{\mathfrak{m}_0}$, 
$$0 \subset (h_1) \subseteq (h_1, h_2) \subseteq (h_1, h_2, h_3) \subseteq \ldots.$$
Assume to the contrary that the chain stabilizes.
Then there is an $N \geq 1$ such that
$$h_N = \sum_{n=1}^{N-1} c_nh_n,$$
with $c_n \in R_{\mathfrak{m}_0}$.
Thus, since $R$ is an integral domain, 
$$g^N = \sum_{n =1}^{N-1}c_ng^n.$$
Furthermore, since $R$ is a subalgebra of the polynomial ring $B$ and $g$ is a monomial, there is some $1 \leq n \leq N-1$ such that
$$c_n = g^{N-n} + b,$$
with $b \in R_{\mathfrak{m}_0}$.
But then $g^{N-n} = c_n - b \in R_{\mathfrak{m}_0}$, contrary to Claim (i).

(iii) Similarly, $\hat{Z}_{\mathfrak{z}_0}$ is nonnoetherian.
\end{proof}

Recall that by monomial, we mean a nonconstant monomial.

\begin{Lemma} \label{contains a non-constant monomial}
Suppose that each monomial in $\hat{Z}$ is divisible (in $B$) by $\sigma$.
If $\mathfrak{p} \in \operatorname{Spec}\hat{Z}$ contains a monomial, then $\mathfrak{p} = \mathfrak{z}_0$.
\end{Lemma}

\begin{proof}
Suppose $\mathfrak{p} \in \operatorname{Spec}\hat{Z}$ contains a monomial.
Then $\sigma$ is in $\mathfrak{p}$ by Lemma \ref{sigma in m}.
Furthermore, there is some $\mathfrak{q} \in \operatorname{Spec}S$ such that $\mathfrak{q} \cap \hat{Z} = \mathfrak{p}$, by Lemma \ref{max ideal}.

Suppose $g$ is a monomial in $\hat{Z}$.
By assumption, there is a monomial $h$ in $B$ such that $g = \sigma h$.
By Lemma \ref{cyclelemma2}.5, $h$ is also in $S$.
Whence $g = \sigma h \in \mathfrak{q}$ since $\sigma \in \mathfrak{p} \subseteq \mathfrak{q}$.
But $g \in \hat{Z}$.
Therefore $g \in \mathfrak{q} \cap \hat{Z} = \mathfrak{p}$.
Since $g$ was arbitrary, $\mathfrak{p}$ contains all monomials in $\hat{Z}$.
\end{proof}

\begin{Remark} \rm{
In Lemma \ref{contains a non-constant monomial}, we assumed that $\sigma$ divides all monomials in $\hat{Z}$.
An example of a dimer algebra with this property is given in Figure \ref{sigma divides all}, where the center is the nonnoetherian ring $Z \cong \hat{Z} = R = k + \sigma S$, and the cycle algebra is the quadric cone, $S = k[xz, xw, yz, yw]$.  
The contraction $\psi: A \to A'$ is cyclic. 
}\end{Remark}

\begin{figure}
$$\begin{array}{ccc}
\xy  0;/r.4pc/:
(-12,6)*+{\text{\scriptsize{$2$}}}="1";(0,6)*+{\text{\scriptsize{$1$}}}="2";(12,6)*+{\text{\scriptsize{$2$}}}="3";
(-12,-6)*+{\text{\scriptsize{$1$}}}="4";(0,-6)*+{\text{\scriptsize{$2$}}}="5";(12,-6)*+{\text{\scriptsize{$1$}}}="6";
(-12,0)*{\cdot}="7";(0,0)*{\cdot}="8";(12,0)*{\cdot}="9";
(-6,6)*{\cdot}="10";(6,6)*{\cdot}="11";(-6,-6)*{\cdot}="12";(6,-6)*{\cdot}="13";
{\ar@[green]_1"2";"10"};{\ar@{..>}"2";"10"};
{\ar_y"10";"1"};{\ar^{}_w"7";"4"};
{\ar@[green]_1"4";"12"};{\ar@{..>}"4";"12"};
{\ar_x"12";"5"};
{\ar@[green]^1"5";"8"};{\ar@{..>}"5";"8"};
{\ar@[green]^1"2";"11"};{\ar@{..>}"2";"11"};
{\ar^x"11";"3"};{\ar^w"9";"6"};
{\ar@[green]^1"6";"13"};{\ar@{..>}"6";"13"};
{\ar^y"13";"5"};
{\ar@[green]^1"3";"9"};{\ar@{..>}"3";"9"};
{\ar@[green]_1"1";"7"};{\ar@{..>}"1";"7"};
{\ar^z"8";"2"};
\endxy
& \ \ \ \stackrel{\psi}{\longrightarrow} \ \ \ &
\xy 0;/r.4pc/:
(-12,6)*+{\text{\scriptsize{$2$}}}="1";(0,6)*+{\text{\scriptsize{$1$}}}="2";(12,6)*+{\text{\scriptsize{$2$}}}="3";
(-12,-6)*+{\text{\scriptsize{$1$}}}="4";(0,-6)*+{\text{\scriptsize{$2$}}}="5";(12,-6)*+{\text{\scriptsize{$1$}}}="6";
{\ar_{y}"2";"1"};{\ar_{w}"1";"4"};{\ar_{x}"4";"5"};{\ar^{z}"5";"2"};{\ar^{x}"2";"3"};{\ar^{w}"3";"6"};{\ar^{y}"6";"5"};
\endxy\\
Q & & Q'
\end{array}$$
\caption{An example where $\sigma$ divides all monomials in $R$. The quivers are drawn on a torus, the contracted arrows are drawn in dotted green, and the arrows are labeled by their $\bar{\tau}_{\psi}$ and $\bar{\tau}$-images.}
\label{sigma divides all}
\end{figure}

\begin{Lemma} \label{a cycle q}
Suppose that there is a monomial in $\hat{Z}$ which is not divisible (in $B$) by $\sigma$.
Let $\mathfrak{m} \in \operatorname{Max}\hat{Z} \setminus \left\{ \mathfrak{z}_0 \right\}$.
Then there is a monomial $g \in \hat{Z} \setminus \mathfrak{m}$ such that $\sigma \nmid g$.
\end{Lemma}

\begin{proof}
Let $\mathfrak{m} \in \operatorname{Max}\hat{Z} \setminus \left\{ \mathfrak{z}_0 \right\}$.

(i) We first claim that there is a monomial in $\hat{Z} \setminus \mathfrak{m}$.
Assume otherwise.
Then
$$\mathfrak{n}_0 \cap \hat{Z} \subseteq \mathfrak{m}.$$
But $\mathfrak{z}_0 := \mathfrak{n}_0 \cap \hat{Z}$ is a maximal ideal by Lemma \ref{max ideal}.
Thus $\mathfrak{z}_0 = \mathfrak{m}$, contrary to assumption.

(ii) We now claim that there is a monomial in $\hat{Z} \setminus \mathfrak{m}$ which is not divisible by $\sigma$.
Indeed, assume to the contrary that every monomial in $\hat{Z}$, which is not divisible by $\sigma$, is in $\mathfrak{m}$.
By assumption, there is a monomial in $\hat{Z}$ that is not divisible by $\sigma$. 
Thus there is at least one monomial in $\mathfrak{m}$.
Therefore $\sigma$ is in $\mathfrak{m}$, by Lemma \ref{sigma in m}.

There is an $\mathfrak{n} \in \operatorname{Max}S$ such that $\mathfrak{n} \cap \hat{Z} = \mathfrak{m}$, by Lemma \ref{max ideal}.
Furthermore, $\sigma \in \mathfrak{n}$ since $\sigma \in \mathfrak{m}$.
Suppose $g \in \hat{Z}$ is a monomial for which $\sigma \mid g$; say $g = \sigma h$ for some monomial $h \in B$.
Then $h \in S$ by Lemma \ref{cyclelemma2}.5.
Whence, $g = \sigma h \in \mathfrak{n}$.
Thus
$$g \in \mathfrak{n} \cap \hat{Z} = \mathfrak{m}.$$
It follows that every monomial in $\hat{Z}$, which \textit{is} divisible by $\sigma$, is also in $\mathfrak{m}$.
Therefore every monomial in $\hat{Z}$ is in $\mathfrak{m}$.
But this contradicts our choice of $\mathfrak{m}$ by Claim (i).
\end{proof}

\begin{figure}
$$\begin{array}{c}
\xy  0;/r.4pc/:
(-30,2.6)*{\cdot}="1";(-24,2.6)*{\cdot}="2";(-18,2.6)*{\cdot}="3";(-12,2.6)*{\cdot}="4";
(-6,2.6)*{\cdot}="5";
(3,2.6)*{\cdot}="6";(9,2.6)*{\cdot}="7";(15,2.6)*{\cdot}="8";(21,2.6)*{\cdot}="9";
(27,2.6)*{\cdot}="10";
(-27,-2.6)*{\cdot}="11";(-21,-2.6)*{\cdot}="12";(-15,-2.6)*{\cdot}="13";
(-9,-2.6)*{\cdot}="14";(-3,-2.6)*{\cdot}="15";
(6,-2.6)*{\cdot}="16";(12,-2.6)*{\cdot}="17";(18,-2.6)*{\cdot}="18";
(24,-2.6)*{\cdot}="19";(30,-2.6)*{\cdot}="20";
{\ar@[blue]"1";"2"};{\ar@[blue]"2";"3"};{\ar@[blue]"3";"4"};{\ar@[blue]"4";"5"};
{\ar@{.}@[blue]"5";"6"};{\ar@[blue]"6";"7"};{\ar@[blue]"7";"8"};{\ar@[blue]"8";"9"};
{\ar@[blue]"9";"10"};
{\ar@[red]"11";"12"};{\ar@[red]"12";"13"};{\ar@[red]"13";"14"};{\ar@[red]"14";"15"};
{\ar@[red]@{.}"15";"16"};{\ar@[red]"16";"17"};{\ar@[red]"17";"18"};{\ar@[red]"18";"19"};
{\ar@[red]"19";"20"};
{\ar@[brown]"20";"10"};{\ar"10";"19"};{\ar"19";"9"};{\ar"9";"18"};{\ar"18";"8"};{\ar"8";"17"};{\ar"17";"7"};{\ar"7";"16"};{\ar"16";"6"};
{\ar"15";"5"};{\ar"5";"14"};{\ar"14";"4"};{\ar"4";"13"};{\ar"13";"3"};{\ar"3";"12"};
{\ar"12";"2"};{\ar"2";"11"};{\ar@[brown]"11";"1"};
\endxy 
\\ 
\text{a column subquiver}\\
\\
\xy  0;/r.4pc/:
(-15,2.6)*{}="0";(-12,2.6)*{\cdot}="1";(-6,2.6)*{\cdot}="2";
(3,2.6)*{\cdot}="3";(9,2.6)*{\cdot}="4";(15,2.6)*{\cdot}="5";
(-18,0)*{\cdot}="6";
(-15,-2.6)*{\cdot}="7";(-9,-2.6)*{\cdot}="8";(-3,-2.6)*{\cdot}="9";
(6,-2.6)*{\cdot}="10";(12,-2.6)*{\cdot}="11";
(15,-2.6)*{}="12";
(18,0)*{\cdot}="13";
{\ar@[blue]@/^.3pc/@{-}"6";"0"};{\ar@[blue]"0";"1"};{\ar@[blue]"1";"2"};
{\ar@{.}"2";"3"};{\ar@[blue]"3";"4"};{\ar@[blue]"4";"5"};
{\ar@[blue]@/^.3pc/"5";"13"};
{\ar@[red]@/_.3pc/"6";"7"};{\ar@[red]"7";"8"};{\ar@[red]"8";"9"};
{\ar@{.}"9";"10"};{\ar@[red]"10";"11"};{\ar@[red]@{-}"11";"12"};
{\ar@[red]@/_.3pc/"12";"13"};
{\ar@/^.3pc/"13";"5"};{\ar"5";"11"};{\ar"11";"4"};{\ar"4";"10"};{\ar"10";"3"};
{\ar"9";"2"};{\ar"2";"8"};{\ar"8";"1"};{\ar"1";"7"};{\ar@/_.3pc/"7";"6"};
\endxy \\
\text{a pillar subquiver}
\end{array}$$
\caption{A column and pillar of a dimer quiver $Q$.
The black interior arrows are arrows of $Q$; the blue and red boundary arrows are paths of length at least one in $Q$; and each interior cycle is a unit cycle of $Q$.
The leftmost and rightmost brown arrows of the column are identified. 
Note that the blue and red bounding paths of the pillar are equal modulo $I$.}
\label{columnsandpillarsfigure}
\end{figure}

\begin{Definition} \label{columnpillardefinition} \rm{
A \textit{column} and \textit{pillar} of a dimer quiver $Q$ are subquivers as shown in Figure \ref{columnsandpillarsfigure}.
}\end{Definition}

\begin{Lemma} \label{columns and pillars}
Let $z = \sum_{i \in Q_0} q_i$ be a central element of $A$ such that for each $i \in Q_0$, $q_i$ is a cycle in $e_iAe_i$.
(In particular, $\overbar{q}_i = \overbar{q}_j$ for each $i,j \in Q_0$.)
If $\sigma \nmid \overbar{q}_i$, then the set of representatives $\tilde{q}_i \in e_ikQe_i$ of the $q_i$ partition $Q$ into columns and pillars.
\end{Lemma}

\begin{proof}
See \cite[Lemmas 4.8.3 and 4.12]{B2}.
\end{proof}

\begin{Lemma} \label{notg}
Let $p \in \mathcal{C}^u$, $q \in \mathcal{C}^v$ be cycles in $Q$ such that $u,v \in \mathbb{Z}^2$ are linearly independent over $\mathbb{R}$. 
If 
$$\overbar{p} \in S \setminus \hat{Z}, \ \ \ \ \overbar{q} \in \hat{Z},$$ 
and $\sigma \nmid \overbar{q}$, then 
$$\overbar{p} \overbar{q} \in \hat{Z}.$$
\end{Lemma}

\begin{proof} 
Since $\overbar{q}$ is a monomial in $\hat{Z}$, for each $j \in Q_0$ there is a cycle $q_j \in e_jAe_j$ such that $\overbar{q}_j = \overbar{q}$, $q_{\operatorname{t}(q)} = q$, and the sum
$$z := \sum_{j \in Q_0}q_j$$
is in the center $Z$ of $A$.
Furthermore, since $\sigma \nmid \overbar{q}$, the set of representatives of the $q_j$ partition $Q$ into columns and pillars, by Lemma \ref{columns and pillars}.

Fix a representative $\tilde{p}$ of $p$ (that is, $\tilde{p} + I = p$), and for each $j \in Q_0$, choose a representative $\tilde{q}_j$ of $q_j$.
By assumption, $u$ and $v$ are linearly independent over $\mathbb{R}$.
Thus, for each $j \in Q_0$, the cycles $\tilde{q}_j$ and $\tilde{p}$ intersect at some vertex $i(j) \in Q_0$.
Factor $\tilde{q}_j$ and $\tilde{p}$ into paths
\begin{align*}
\tilde{q}_j & = q_{j2}e_{i(j)}q_{j1}\\
\tilde{p} & = p_{j2}e_{i(j)}p_{j1},
\end{align*}
and consider the cycles
\begin{equation} \label{cycler}
r_j := q_{j2}p_{j1}p_{j2}q_{j1} + I \in e_jAe_j.
\end{equation}
Note that 
$$\overbar{r}_j = \overbar{p}\overbar{q}_j = \overbar{p} \overbar{q}.$$
Thus, to prove the lemma, it suffices to show that the element 
$$r := \sum_{j \in Q_0}r_j$$ 
is in the center $Z$ of $A$.

(a) We first claim that $r$ is independent of the choice of representatives $\tilde{q}_j$ of $q_j$.
Fix $j \in Q_0$, and suppose $\tilde{q}_j$ and $\tilde{q}'_j$ are two representatives of $q_j$.
By Lemma \ref{columns and pillars}, it suffices to suppose that $\tilde{q}_j$ and $\tilde{q}_j'$ bound a pillar, and $\tilde{p}$ intersects this pillar as shown in Figure \ref{claima}.
Factor $\tilde{p}$, $\tilde{q}_j$, $\tilde{q}'_j$ into paths
$$\tilde{p} = p_2bp_1, \ \ \ \ \tilde{q}_j = q_5q_3q_2q_1, \ \ \ \ \tilde{q}'_j = q_5q'_4q'_3q'_2q_1.$$
Then
\begin{align*}
q_5 q_3bp_1p_2q_2q_1 & \stackrel{\textsc{(i)}}{\equiv} q_5q'_4q'_3cbp_1p_2q_2q_1\\
& \equiv q_5q'_4\sigma_{\operatorname{h}(p_1)}p_1p_2q_2q_1\\
& \stackrel{\textsc{(ii)}}{\equiv} q_5q'_4p_1p_2\sigma_{\operatorname{t}(p_2)}q_2q_1\\
& \equiv q_5q'_4p_1p_2 bq'_3cq_2q_1\\
& \stackrel{\textsc{(iii)}}{\equiv} q_5q'_4p_1p_2bq'_2q_1,
\end{align*}
where (\textsc{i}) and (\textsc{iii}) hold by the dimer relations since each arrow in the interior of the column is an arrow in the quiver, and (\textsc{ii}) holds by Lemma \ref{sigmalemma}.
Therefore the two representatives $\tilde{q}_j$, $\tilde{q}'_j$ of $q_j$ define the same cycle $r_j \in e_jAe_j$ in (\ref{cycler}).

(b) We now claim that $r$ is in $Z$.
Since $r$ is a sum of cycles, $r$ trivially commutes with the vertex idempotents.
Thus, to show that $r$ is in $Z$, it suffices to show that $ra = ar$ for each arrow $a \in Q_1$.

Fix an arrow $a$.
Set $j := \operatorname{t}(a)$ and $k := \operatorname{h}(a)$.

If $a$ is a leftmost arrow subpath of a representative $\tilde{r}_k$ of $r_k$, then there is a representative $\tilde{r}_j$ of $r_j$ that is a cyclic permutation of $\tilde{r}_k$, by Claim (a).
Whence
$$ra = r_k a = a r_j = ar.$$

So suppose $a$ is not a leftmost arrow subpath of any representative of $r_k$.
There are three cases to consider, shown in Figure \ref{columnscase}.

Case (i): Suppose $\tilde{q}_j$ and $\tilde{q}_k$ bound a column, and $\tilde{p}$ intersects this column as shown in Figure \ref{columnscase}.i.
Factor $\tilde{p}$, $\tilde{q}_j$, $\tilde{q}_k$ into paths
$$\tilde{p} = p_2bp_1, \ \ \ \ \tilde{q}_j = q_3q_2q_1, \ \ \ \ \tilde{q}_k = q'_2q'_1.$$
Then
\begin{align*}
(q'_2(bp_1p_2)q'_1)a & \stackrel{\textsc{(i)}}{\equiv} q'_2bp_1p_2bq_2q_1\\
& \stackrel{\textsc{(i)}}{\equiv} a(q_3(p_1p_2b)q_2q_1),
\end{align*}
where (\textsc{i}) and (\textsc{ii}) hold by the dimer relations since each arrow in the interior of the column is an arrow in the quiver.

Case (ii): Suppose $\tilde{q}_j$ and $\tilde{q}_k$ bound a column, and $\tilde{p}$ intersects this column as shown in Figure \ref{columnscase}.ii.
Factor $\tilde{p}$, $\tilde{q}_j$, $\tilde{q}_k$ into paths
$$\tilde{p} = p_2cp_1, \ \ \ \ \tilde{q}_j = q_3q_2q_1, \ \ \ \ \tilde{q}_k = q'_2q'_1.$$
Thus,
\begin{align*}
(q'_2(p_1p_2c)q'_1)a & \stackrel{\textsc{(i)}}{\equiv} q'_2p_1p_2cbq_2q_1\\
& \equiv q'_2p_1p_2 \sigma_{\operatorname{t}(q_2)} q_1\\
& \stackrel{\textsc{(ii)}}{\equiv} q'_2\sigma_{\operatorname{h}(p_1)}p_1p_2q_1\\
& \equiv q'_2 bq_2cp_1p_2q_1\\
& \stackrel{\textsc{(iii)}}{\equiv} a(q_3q_2(cp_1p_2)q_1),
\end{align*}
where (\textsc{i}) and (\textsc{iii}) hold by the dimer relations since each arrow in the interior of the column is an arrow in the quiver, and (\textsc{ii}) holds by Lemma \ref{sigmalemma}.

Case (iii): Suppose $\tilde{q}_j$ and $\tilde{q}_k$ bound a pillar, and $\tilde{p}$ intersects $\tilde{q}_j$ as shown in Figure \ref{columnscase}.iii.
Factor $\tilde{p}$, $\tilde{q}_j$, $\tilde{q}_k$ into paths
$$\tilde{p} = p_2p_1, \ \ \ \ \tilde{q}_j = q_4q_3q_2q_1, \ \ \ \ \tilde{q}_k = q'_4q_3q_2q'_1.$$ 
Whence,
\begin{align*}
(q'_4q_3(p_1p_2)q_2q'_1)a & \stackrel{\textsc{(i)}}{\equiv} q'_4q_3p_1p_2q_2\sigma_{\operatorname{t}(q_2)} q_1\\
& \stackrel{\textsc{(ii)}}{\equiv} q'_4\sigma_{\operatorname{h}(q_3)}q_3p_1p_2q_2q_1\\
& \stackrel{\textsc{(iii)}}{\equiv} a(q_4q_3(p_1p_2)q_2q_1)
\end{align*}
where (\textsc{i}) and (\textsc{iii}) hold by the dimer relations since each arrow in the interior of the pillar is an arrow in the quiver, and (\textsc{ii}) holds by Lemma \ref{sigmalemma}.

Therefore $ra = ar$ holds in each case.
\end{proof}

\begin{figure}
$$\xy  0;/r.4pc/:
(-33,0)*+{\text{\scriptsize{$j$}}}="1";(-21,0)*{\cdot}="2";(-15,2.6)*{\cdot}="3";
(-6,2.6)*{\cdot}="4";(0,2.6)*{\cdot}="5";(6,2.6)*{\cdot}="6";
(15,2.6)*{\cdot}="7";(21,2.6)*{\cdot}="8";
(-18,-2.6)*{\cdot}="10";(-12,-2.6)*{\cdot}="11";
(-3,-2.6)*{\cdot}="12";(3,-2.6)*{\cdot}="13";
(12,-2.6)*{\cdot}="14";(18,-2.6)*{\cdot}="15";(24,0)*{\cdot}="16";
(0,14)*{\cdot}="17";
(0,-14)*{\cdot}="19";
(36,0)*+{\text{\scriptsize{$j$}}}="20";
(-18,2.6)*{}="2.5";(21,-2.6)*{}="15.5";
{\ar@[blue]^{q_1}"1";"2"};{\ar@{-}@[blue]@/^.3pc/"2";"2.5"};{\ar@[blue]"2.5";"3"};
{\ar@[blue]"4";"5"};{\ar@[blue]"5";"6"};{\ar@[blue]"7";"8"};
{\ar@[red]"10";"11"};{\ar@[red]_{q'_3}"12";"13"};{\ar@[red]"14";"15"};
{\ar@{-}@[red]"15";"15.5"};{\ar@[red]@/_.3pc/"15.5";"16"};
{\ar@/^.3pc/"16";"8"};{\ar@[blue]@/^.3pc/"8";"16"};
{\ar"8";"15"};{\ar"15";"7"};{\ar"7";"14"};
{\ar"6";"13"};{\ar@[orange]|-b"13";"5"};{\ar|-c"5";"12"};{\ar"12";"4"};
{\ar"11";"3"};{\ar"3";"10"};
{\ar@[red]@/_.3pc/"2";"10"};{\ar@/_.3pc/"10";"2"};
{\ar@[orange]_{p_1}"19";"13"};{\ar@[orange]_{p_2}"5";"17"};
{\ar@{.}@[blue]^{q_2}"3";"4"};{\ar@{.}@[blue]^{q_3}"6";"7"};
{\ar@{.}@[red]_{q'_2}"11";"12"};{\ar@{.}@[red]_{q'_4}"13";"14"};
{\ar@[blue]_{q_5}"16";"20"};
\endxy$$
\caption{The paths $\tilde{p} = p_2bp_1$, $\tilde{q}_j = q_5q_3q_2q_1$, and $\tilde{q}'_j = q_5q'_4q'_3q'_2q_1$ for Claim (a) in the proof of Lemma \ref{notg}.}
\label{claima}
\end{figure}

\begin{figure}
$$\begin{array}{ccc}
(i): & & \xy  0;/r.4pc/:
(-27,2.6)*{\cdot}="1";(-21,2.6)*{\cdot}="2";(-15,2.6)*{\cdot}="3";
(-6,2.6)*{\cdot}="4";(0,2.6)*{\cdot}="5";(6,2.6)*{\cdot}="6";
(15,2.6)*{\cdot}="7";(21,2.6)*{\cdot}="8";
(-24,-2.6)*{\cdot}="9";(-18,-2.6)*{\cdot}="10";(-12,-2.6)*{\cdot}="11";
(-3,-2.6)*{\cdot}="12";(3,-2.6)*{\cdot}="13";
(12,-2.6)*{\cdot}="14";(18,-2.6)*{\cdot}="15";(24,-2.6)*{\cdot}="16";
(0,14)*{\cdot}="17";
(0,-14)*{\cdot}="19";
{\ar@[blue]"1";"2"};{\ar@[blue]"2";"3"};{\ar@[blue]"4";"5"};{\ar@[blue]"5";"6"};{\ar@[blue]"7";"8"};
{\ar@[red]"9";"10"};{\ar@[red]"10";"11"};{\ar@[red]_{q_2}"12";"13"};{\ar@[red]"14";"15"};{\ar@[red]"15";"16"};
{\ar@[brown]_a"16";"8"};{\ar"8";"15"};{\ar"15";"7"};{\ar"7";"14"};
{\ar"6";"13"};{\ar@[orange]|-b"13";"5"};{\ar|-c"5";"12"};{\ar"12";"4"};
{\ar"11";"3"};{\ar"3";"10"};{\ar"10";"2"};{\ar"2";"9"};{\ar@[brown]^a"9";"1"};
{\ar@[orange]_{p_1}"19";"13"};{\ar@[orange]_{p_2}"5";"17"};
{\ar@{.}@[blue]^{q'_1}"3";"4"};{\ar@{.}@[blue]^{q'_2}"6";"7"};
{\ar@{.}@[red]_{q_1}"11";"12"};{\ar@{.}@[red]_{q_3}"13";"14"};
\endxy
\\ \\
(ii): & & \xy  0;/r.4pc/:
(-27,2.6)*{\cdot}="1";(-21,2.6)*{\cdot}="2";(-15,2.6)*{\cdot}="3";
(-6,2.6)*{\cdot}="4";(0,2.6)*{\cdot}="5";(6,2.6)*{\cdot}="6";
(15,2.6)*{\cdot}="7";(21,2.6)*{\cdot}="8";
(-24,-2.6)*{\cdot}="9";(-18,-2.6)*{\cdot}="10";(-12,-2.6)*{\cdot}="11";
(-3,-2.6)*{\cdot}="12";(3,-2.6)*{\cdot}="13";
(12,-2.6)*{\cdot}="14";(18,-2.6)*{\cdot}="15";(24,-2.6)*{\cdot}="16";
(0,14)*{\cdot}="17";
(0,-14)*{\cdot}="19";
{\ar@[blue]"1";"2"};{\ar@[blue]"2";"3"};{\ar@[blue]"4";"5"};{\ar@[blue]"5";"6"};{\ar@[blue]"7";"8"};
{\ar@[red]"9";"10"};{\ar@[red]"10";"11"};{\ar@[red]_{q_2}"12";"13"};{\ar@[red]"14";"15"};{\ar@[red]"15";"16"};
{\ar@[brown]_a"16";"8"};{\ar"8";"15"};{\ar"15";"7"};{\ar"7";"14"};
{\ar"6";"13"};{\ar|-b"13";"5"};{\ar@[orange]|-c"5";"12"};{\ar"12";"4"};
{\ar"11";"3"};{\ar"3";"10"};{\ar"10";"2"};{\ar"2";"9"};{\ar@[brown]^a"9";"1"};
{\ar@[orange]_{p_1}"17";"5"};{\ar@[orange]_{p_2}"12";"19"};
{\ar@{.}@[blue]^{q'_1}"3";"4"};{\ar@{.}@[blue]^{q'_2}"6";"7"};
{\ar@{.}@[red]_{q_1}"11";"12"};{\ar@{.}@[red]_{q_3}"13";"14"};
\endxy
\\ \\
(iii): & &
\xy  0;/r.4pc/:
(-27,2.6)*{\cdot}="1";(-21,2.6)*{\cdot}="2";(-15,2.6)*{\cdot}="3";
(-12,0)*{\cdot}="4";(0,0)*{\cdot}="5";(12,0)*{\cdot}="6";
(18,2.6)*{\cdot}="8";
(-24,-2.6)*{\cdot}="9";(-18,-2.6)*{\cdot}="10";
(15,-2.6)*{\cdot}="15";(21,-2.6)*{\cdot}="16";
(0,14)*{\cdot}="17";
(0,-14)*{\cdot}="19";
(24,2.6)*{\cdot}="12";(27,-2.6)*{\cdot}="13";
(15,2.6)*{}="7";(-15,-2.6)*{}="11";
{\ar@[blue]"1";"2"};{\ar@[blue]^{q'_{1}}"2";"3"};{\ar@[blue]@/^.3pc/"3";"4"};
{\ar@[blue]^{q_{2}}"4";"5"};{\ar@[blue]^{q_{3}}"5";"6"};
{\ar@[red]"9";"10"};{\ar@{-}@[red]_{q_1}"10";"11"};{\ar@[red]@/_.3pc/"11";"4"};
{\ar@/_.3pc/"15";"6"};{\ar_{q_4}@[red]"15";"16"};
{\ar"16";"8"};{\ar"8";"15"};{\ar@/_.3pc/@[red]"6";"15"};
{\ar@[blue]@{-}@/^.3pc/"6";"7"};{\ar@[blue]^{q'_4}"7";"8"};
{\ar@/^.3pc/"4";"3"};
{\ar"3";"10"};{\ar"10";"2"};{\ar"2";"9"};{\ar@[brown]^a"9";"1"};
{\ar@[orange]_{p_2}"5";"17"};{\ar@[orange]_{p_1}"19";"5"};
{\ar@[blue]"8";"12"};{\ar"12";"16"};{\ar@[red]"16";"13"};{\ar@[brown]_a"13";"12"};
\endxy
\end{array}$$
\caption{The three cases for Claim (b) in the proof of Lemma \ref{notg}.}
\label{columnscase}
\end{figure}

\begin{Lemma} \label{iveupyet}
Let $p \in \mathcal{C}^u$, $q \in \mathcal{C}^v$ be cycles in $Q$ such that $u,v \in \mathbb{Z}^2$ are linearly dependent over $\mathbb{R}$.
If 
$$\overbar{p} \in S \setminus \hat{Z}, \ \ \ \ \overbar{q} \in \hat{Z},$$ 
and $\sigma$ does not divide $\overbar{p}$ or $\overbar{q}$, then there is a cycle $r$ and integers $m,n \geq 1$ such that
$$r^n = q, \ \ \ \ \overbar{r} \in \hat{Z}, \ \ \ \text{ and } \ \ \ \overbar{p} \overbar{r}^m \in \hat{Z}.$$
\end{Lemma}

\begin{proof}
(i) First suppose $u = (0,0)$.
Since $\overbar{p} \not \in \hat{Z}$, $p$ is not a vertex.
Thus, there is some $\ell \geq 1$ such that $\overbar{p} = \sigma^{\ell}$, by Lemma \ref{cyclelemma}.1.
Whence $\sigma \mid \overbar{p}$, contrary to assumption. 

(ii) Assume to the contrary that there are positive integers $n_1,n_2 \geq 1$ such that $n_1u = n_2v$.
By assumption, $\sigma \nmid \overbar{p}$.
Thus there is a simple matching $D$ such that $x_D \nmid \overbar{p}$.
Whence $\sigma \nmid \overbar{p}^{n_1}$.
Similarly, $\sigma \nmid \overbar{q}^{n_2}$.
Consequently,
\begin{equation} \label{cold}
\overbar{p}^{n_1} \stackrel{\textsc{(i)}}{=} \overbar{q}^{n_2} \in \hat{Z} \stackrel{\textsc{(ii)}}{\subseteq} R,
\end{equation}
where (\textsc{i}) holds by Lemma \ref{cyclelemma}.2, and (\textsc{ii}) holds by \cite[Theorem 1.1.3]{B6}.

Now if a monomial $g \in S$ is not in $R$ and $\sigma \nmid g$, then $g^{\ell}$ is also not in $R$ for each $\ell \geq 1$, by Lemma \ref{cyclelemma2}.6.
Thus (\ref{cold}) implies that $\overbar{p}$ is in $R$.
Therefore $\overbar{p}$ is in $\hat{Z}$ since $\sigma \nmid \overbar{p}$, by Lemma \ref{cyclelemma2}.2.
But this contradicts our choice of $\overbar{p}$.

(iii) Finally, suppose there are positive integers $n_1,n_2 \geq 1$ such that $n_1u = -n_2v$.
Let $\hat{v} \in \mathbb{Z}^2$ be the vector of minimal length (with respect to the standard $\mathbb{R}^2$ metric) such that $v = n_3\hat{v}$ for some $n_3 \geq 1$.
Then there is a cycle $r \in \mathcal{C}^{\hat{v}}$ such that $\sigma \nmid \overbar{r}$, by \cite[Corollary 3.9]{B1}.
Thus $\overbar{r} \in \hat{Z}$ since $\overbar{q} \in \hat{Z}$, by Lemma \ref{cyclelemma2}.6.
Furthermore, setting $m := \frac{n_2n_3}{n_1}$ we have
$$u = - m \hat{v}.$$
Therefore there is some $\ell \geq 1$ such that
$$\overbar{p} \overbar{r}^m \stackrel{\textsc{(i)}}{=} \sigma^{\ell} \in \hat{Z},$$
where (\textsc{i}) holds by Lemma \ref{cyclelemma}.1.
\end{proof}

Recall the subsets (\ref{U U*}) of $\operatorname{Max}S$ and the morphisms (\ref{max surjective}). 

\begin{Proposition} \label{n in maxs}
Let $\mathfrak{n} \in \operatorname{Max}S$.
Then
\begin{equation*} \label{n in maxs1}
\mathfrak{n} \cap \hat{Z} \not = \mathfrak{z}_0 \ \ \text{ if and only if } \ \ \hat{Z}_{\mathfrak{n} \cap \hat{Z}} = S_{\mathfrak{n}},
\end{equation*}
and
\begin{equation*} \label{n in maxs2}
\mathfrak{n} \cap R \not = \mathfrak{m}_0 \ \ \text{ if and only if } \ \ R_{\mathfrak{n} \cap R} = S_{\mathfrak{n}}.
\end{equation*}
Consequently,
$$\kappa_{S/\hat{Z}}(U_{S/\hat{Z}}) = \operatorname{Max}\hat{Z} \setminus \{ \mathfrak{z}_0 \} \ \ \ \text{ and } \ \ \ \kappa_{S/R}(U_{S/R}) = \operatorname{Max}R \setminus \{ \mathfrak{m}_0 \}.$$
\end{Proposition}

\begin{proof} \ \\
\indent (i) Set $\mathfrak{m} := \mathfrak{n} \cap \hat{Z}$, and suppose $\mathfrak{m} \not = \mathfrak{z}_0$.
We first claim that 
\begin{equation} \label{S subseteq}
S \subset \hat{Z}_{\mathfrak{m}}.
\end{equation}

Consider $g \in S \setminus \hat{Z}$. 
By \cite[Proposition 5.14]{B2}, $S$ is generated by $\sigma$ and a set of monomials in $B$ not divisible by $\sigma$. 
Furthermore, $\sigma$ is in $\hat{Z}$. 
It thus suffices to suppose that $g$ is a monomial which is not divisible by $\sigma$; let $p$ be a cycle for which $\overbar{p} = g$.

(i.a) First suppose $\sigma$ does not divide all monomials in $\hat{Z}$.
By Lemma \ref{a cycle q}, there is a nontrivial cycle $q \in A$ such that
$$\overbar{q} \in \hat{Z} \setminus \mathfrak{m} \ \ \ \text{ and } \ \ \ \sigma \nmid \overbar{q}.$$
Let $u,v \in \mathbb{Z}^2$ be such that
$$p \in \mathcal{C}^u \ \ \ \text{ and } \ \ \ q \in \mathcal{C}^v.$$

If $u,v$ are linearly independent over $\mathbb{R}$, then $\overbar{p} \overbar{q}$ is in $\hat{Z}$, by Lemma \ref{notg}.
Whence
\begin{equation*} \label{p=2}
g = \overbar{p} = (\overbar{p} \overbar{q}) \overbar{q}^{-1} \in \hat{Z}_{\mathfrak{m}}.
\end{equation*}
If instead $u,v$ are linearly dependent over $\mathbb{R}$, then there is a cycle $r$ and integers $m,n \geq 1$ such that
$$r^n = q, \ \ \ \ \overbar{r} \in \hat{Z}, \ \ \ \ \overbar{p} \overbar{r}^m \in \hat{Z},$$
by Lemma \ref{iveupyet}.
Furthermore, $\overbar{r} \not \in \mathfrak{m}$ since $\overbar{r}^n = \overbar{q}$ and $\overbar{q} \not \in \mathfrak{m}$.
Thus,
\begin{equation*} \label{notyet}
g = \overbar{p} = (\overbar{p} \overbar{r}^m) \overbar{r}^{-m} \in \hat{Z}_{\mathfrak{m}}.
\end{equation*}
Therefore, in either case, $g$ is in the localization $\hat{Z}_{\mathfrak{m}}$. 
Consequently, (\ref{S subseteq}) holds if $\sigma$ does not divide all monomials in $\hat{Z}$.

(i.b) Now suppose $\sigma$ divides all monomials in $\hat{Z}$.
Then $\mathfrak{m}$ does not contain any monomials since $\mathfrak{m} \not = \mathfrak{z}_0$, by Lemma \ref{contains a non-constant monomial}.
In particular, $\sigma \not \in \mathfrak{m}$.
By Lemma \ref{cyclelemma2}.4, there is an $n \geq 0$ such that $g \sigma^n \in \hat{Z}$.
Thus
$$g = (g \sigma^n) \sigma^{-n} \in \hat{Z}_{\mathfrak{m}}.$$ 
Therefore (\ref{S subseteq}) also holds if $\sigma$ divides all monomials in $\hat{Z}$.

(ii) Denote by $\tilde{\mathfrak{m}} := \mathfrak{m} \hat{Z}_{\mathfrak{m}}$ the maximal ideal of $\hat{Z}_{\mathfrak{m}}$.
Then, since $\hat{Z} \subset S$, we have
$$\hat{Z}_{\mathfrak{m}} = \hat{Z}_{\tilde{\mathfrak{m}} \cap \hat{Z}} \subseteq S_{\tilde{\mathfrak{m}} \cap S} \stackrel{\textsc{(i)}}{\subseteq} (\hat{Z}_{\mathfrak{m}})_{\tilde{\mathfrak{m}} \cap \hat{Z}_{\mathfrak{m}}} = (\hat{Z}_{\mathfrak{m}})_{\mathfrak{m}\hat{Z}_{\mathfrak{m}}} = \hat{Z}_{\mathfrak{m}},$$
where (\textsc{i}) holds by (\ref{S subseteq}).
Therefore
$$S_{\mathfrak{n}} = \hat{Z}_{\mathfrak{m}}.$$

(iii) Now suppose $\mathfrak{n} \cap R \not = \mathfrak{m}_0$.
We claim that $R_{\mathfrak{n} \cap R} = S_{\mathfrak{n}}$.

Since $\mathfrak{n} \cap R \not = \mathfrak{m}_0$, there is a monomial $g \in R \setminus \mathfrak{n}$. 
Thus there is some $n \geq 1$ such that $g^n \in \hat{Z}$, by Lemma \ref{cyclelemma2}.3. 
Furthermore, $g^n \not \in \mathfrak{n}$ since $\mathfrak{n}$ is a prime ideal.
Consequently,
$$g^n \in \hat{Z} \setminus (\mathfrak{n} \cap \hat{Z}).$$
Whence
\begin{equation} \label{today}
\mathfrak{n} \cap \hat{Z} \not = \mathfrak{z}_0.
\end{equation}
Therefore
$$S_{\mathfrak{n}} \stackrel{\textsc{(i)}}{=} \hat{Z}_{\mathfrak{n} \cap \hat{Z}} \stackrel{\textsc{(ii)}}{\subseteq} R_{\mathfrak{n} \cap R} \subseteq S_{\mathfrak{n}},$$
where (\textsc{i}) holds by  (\ref{today}) and Claim (i), and (\textsc{ii}) holds by (\ref{ZKA}). 
It follows that $R_{\mathfrak{n} \cap R} = S_{\mathfrak{n}}$.

(iv) Finally, we claim that
$$\hat{Z}_{\mathfrak{z}_0} \not = S_{\mathfrak{n}_0} \ \ \text{ and } \ \ R_{\mathfrak{m}_0} \not = S_{\mathfrak{n}_0}.$$
These inequalities hold since the local algebras
$\hat{Z}_{\mathfrak{z}_0}$ and $R_{\mathfrak{m}_0}$ are nonnoetherian by Proposition \ref{local nonnoetherian}, whereas $S_{\mathfrak{n}}$ is noetherian by Lemma \ref{S noetherian}.
\end{proof}

\begin{Lemma} \label{n cap Z = n' cap Z}
Let $\mathfrak{q}$ and $\mathfrak{q}'$ be prime ideals of $S$.
Then
$$\mathfrak{q} \cap \hat{Z} = \mathfrak{q}' \cap \hat{Z} \ \ \text{ if and only if } \ \ \mathfrak{q} \cap R = \mathfrak{q}' \cap R.$$
\end{Lemma}

\begin{proof}
(i) Suppose $\mathfrak{q} \cap \hat{Z} = \mathfrak{q}' \cap \hat{Z}$, and let $s \in \mathfrak{q} \cap R$.
Then $s \in R$.
Whence there is some $n \geq 1$ such that $s^n \in \hat{Z}$, by Lemma \ref{cyclelemma2}.3. 
Thus
$$s^n \in \mathfrak{q} \cap \hat{Z} = \mathfrak{q}' \cap \hat{Z}.$$
Therefore $s^n \in \mathfrak{q}'$.
Thus $s \in \mathfrak{q}'$ since $\mathfrak{q}'$ is prime.
Consequently, $s \in \mathfrak{q}' \cap R$.
Therefore $\mathfrak{q} \cap R \subseteq \mathfrak{q}' \cap R$.
Similarly, $\mathfrak{q} \cap R \supseteq \mathfrak{q}' \cap R$.

(ii) Now suppose $\mathfrak{q} \cap R = \mathfrak{q}' \cap R$, and let $s \in \mathfrak{q} \cap \hat{Z}$.
Then $s \in \hat{Z} \subseteq R$.
Thus
$$s \in \mathfrak{q} \cap R = \mathfrak{q}' \cap R.$$
Whence $s \in \mathfrak{q}' \cap \hat{Z}$.
Therefore $\mathfrak{q} \cap \hat{Z} \subseteq \mathfrak{q}' \cap \hat{Z}$.
Similarly, $\mathfrak{q} \cap \hat{Z} \supseteq \mathfrak{q}' \cap \hat{Z}$.
\end{proof}

\begin{Proposition} \label{coincide prop}
The subsets $U_{S/\hat{Z}}$ and $U_{S/R}$ of $\operatorname{Max}S$ are equal.
\end{Proposition}

\begin{proof}
(i) We first claim that
$$U_{S/\hat{Z}} \subseteq U_{S/R}.$$
Indeed, suppose $\mathfrak{n} \in U_{S/\hat{Z}}$.
Then since $\hat{Z} \subseteq R \subseteq S$, we have
$$S_{\mathfrak{n}} = \hat{Z}_{\mathfrak{n} \cap \hat{Z}} \subseteq R_{\mathfrak{n} \cap R} \subseteq S_{\mathfrak{n}}.$$
Thus
$$R_{\mathfrak{n} \cap R} = S_{\mathfrak{n}}.$$
Therefore $\mathfrak{n} \in U_{S/R}$, proving our claim.

(ii) We now claim that
$$U_{S/R} \subseteq U_{S/\hat{Z}}.$$
Let $\mathfrak{n} \in U_{S/R}$.
Then $R_{\mathfrak{n} \cap R} = S_{\mathfrak{n}}$.
Thus by Proposition \ref{n in maxs},
$$\mathfrak{n} \cap R \not = \mathfrak{n}_0 \cap R.$$
Therefore by Lemma \ref{n cap Z = n' cap Z},
$$\mathfrak{n} \cap \hat{Z} \not = \mathfrak{n}_0 \cap \hat{Z}.$$
But then again by Proposition \ref{n in maxs},
$$\hat{Z}_{\mathfrak{n} \cap \hat{Z}} = S_{\mathfrak{n}}.$$
Whence $\mathfrak{n} \in U_{S/\hat{Z}}$, proving our claim.
\end{proof}

We denote the complement of a set $W \subseteq \operatorname{Max}S$ by $W^c$. 

\begin{Theorem} \label{isolated sing}
The following subsets of $\operatorname{Max}S$ are open, dense, and coincide:
\begin{equation} \label{coincide}
\begin{array}{c}
U^*_{S/\hat{Z}} = U_{S/\hat{Z}} = U^*_{S/R} = U_{S/R}\\
 = \kappa_{S/\hat{Z}}^{-1}(\operatorname{Max}\hat{Z} \setminus \left\{ \mathfrak{z}_0 \right\} ) = \kappa_{S/R}^{-1}\left(\operatorname{Max}R \setminus \left\{ \mathfrak{m}_0 \right\} \right)\\
= \mathcal{Z}_S(\mathfrak{z}_0S)^c = \mathcal{Z}_S(\mathfrak{m}_0S)^c.
\end{array}
\end{equation}
In particular, $\hat{Z}$ and $R$ are locally noetherian at all points of $\operatorname{Max}\hat{Z}$ and $\operatorname{Max}R$ except at $\mathfrak{z}_0$ and $\mathfrak{m}_0$.  
\end{Theorem}

\begin{proof}
For brevity, set $\mathcal{Z}(I) := \mathcal{Z}_S(I)$.

(i) We first show the equalities of the top two lines of (\ref{coincide}).

By Proposition \ref{coincide prop}, $U_{S/R} = U_{S/\hat{Z}}$.

By Lemma \ref{S noetherian}, $S$ is noetherian.
Thus for each $\mathfrak{n} \in \operatorname{Max}S$, the localization $S_{\mathfrak{n}}$ is noetherian.
Therefore, by Proposition \ref{n in maxs}, 
$$U^*_{S/\hat{Z}} = U_{S/\hat{Z}} \ \ \text{ and } \ \ U^*_{S/R} = U_{S/R}.$$
Moreover, again by Proposition \ref{n in maxs},
$$U_{S/\hat{Z}} = \kappa_{S/\hat{Z}}^{-1}( \operatorname{Max}\hat{Z} \setminus \left\{ \mathfrak{z}_0 \right\}) \ \ \text{ and } \ \ U_{S/R} = \kappa_{S/R}^{-1}\left( \operatorname{Max}R \setminus \left\{ \mathfrak{m}_0 \right\} \right).$$

(ii) We now claim that the complement of $U_{S/R} \subset \operatorname{Max}S$ is the zero locus $\mathcal{Z}(\mathfrak{m}_0S)$.
Suppose $\mathfrak{n} \in \mathcal{Z}(\mathfrak{m}_0S)$; then $\mathfrak{m}_0S \subseteq \mathfrak{n}$.
Whence,
\begin{equation*} \label{m0 = m0 cap R}
\mathfrak{m}_0 \subseteq \mathfrak{m}_0S \cap R \subseteq \mathfrak{n} \cap R.
\end{equation*}
Thus $\mathfrak{n} \cap R = \mathfrak{m}_0$ since $\mathfrak{m}_0$ is a maximal ideal of $R$.
But then $\mathfrak{n} \not \in U_{S/R}$ by Claim (i).
Therefore $U^c_{S/R} \supseteq \mathcal{Z}(\mathfrak{m}_0S)$.

Conversely, suppose $\mathfrak{n} \in U_{S/R}^c$.
Then $\kappa_{S/R}(\mathfrak{n}) \not \in \kappa_{S/R}(U_{S/R})$, by the definition of $U_{S/R}$.
Thus $\mathfrak{n} \cap R = \kappa_{S/R}(\mathfrak{n}) = \mathfrak{m}_0$, by Claim (i).
Whence,
$$\mathfrak{m}_0S = (\mathfrak{n} \cap R)S \subseteq \mathfrak{n},$$
and so $\mathfrak{n} \in \mathcal{Z}(\mathfrak{m}_0S)$.
Therefore $U^c_{S/R} \subseteq \mathcal{Z}(\mathfrak{m}_0S)$.

(iii) Finally, we claim that the subsets (\ref{coincide}) are open dense.
Indeed, there is a maximal ideal of $R$ distinct from $\mathfrak{m}_0$, and $\kappa_{S/R}$ is surjective by Lemma \ref{max ideal}.
Thus the set $\kappa_{S/R}^{-1}\left( \operatorname{Max}R \setminus \left\{ \mathfrak{m}_0 \right\} \right)$ is nonempty.
Moreover, this set equals the open set $\mathcal{Z}(\mathfrak{m}_0S)^c$, by Claim (ii).
But $S$ is an integral domain.
Therefore $\mathcal{Z}(\mathfrak{m}_0S)^c$ is dense since it is nonempty and open. 
\end{proof}

\begin{Theorem} \label{generically noetherian}
The center $Z$, reduced center $\hat{Z}$, and ghor center $R$ of $A$ each have Krull dimension 3,
$$\operatorname{dim}Z = \operatorname{dim}\hat{Z} = \operatorname{dim}R = \operatorname{dim}S = 3.$$
Furthermore, the fraction fields of $\hat{Z}$, $R$, and $S$ coincide,
\begin{equation} \label{function fields}
\operatorname{Frac}\hat{Z} = \operatorname{Frac}R = \operatorname{Frac}S.
\end{equation}
\end{Theorem}

\begin{proof}
Recall that $S$ is of finite type by Lemma \ref{S noetherian}, and $\hat{Z} \subseteq R \subseteq S$ are integral domains since they are subalgebras of the polynomial ring $B$.

The sets $U_{S/\hat{Z}}$ and $U_{S/R}$ are nonempty, by Theorem \ref{isolated sing}.
Thus $\hat{Z}$, $R$, and $S$ have equal fraction fields \cite[Lemma 2.4]{B5}; and equal Krull dimensions \cite[Theorem 2.5.4]{B5}.
In particular, $\operatorname{dim} \hat{Z} = \operatorname{dim}R = \operatorname{dim}S = 3$, by Lemma \ref{Jess}.
Finally, each prime $\mathfrak{p} \in \operatorname{Spec}Z$ contains the nilradical $\operatorname{nil}Z$, and thus $\operatorname{dim}Z = \operatorname{dim}\hat{Z}$.
\end{proof}

Recall that the reduction $X_{\operatorname{red}}$ of a scheme $X$, that is, its reduced induced scheme structure, is the closed subspace of $X$ associated to the sheaf of ideals $\mathcal{I}$, where for each open set $U \subset X$,
$$\mathcal{I}(U) := \left\{ f \in \mathcal{O}_X(U) \ | \ f(\mathfrak{p}) = 0 \ \text{ for all } \ \mathfrak{p} \in U \right\}.$$
$X_{\operatorname{red}}$ is the unique reduced scheme whose underlying topological space equals that of $X$.
If $R := \mathcal{O}_X(X)$, then $\mathcal{O}_{X_{\operatorname{red}}}(X_{\operatorname{red}}) = R/\operatorname{nil}R$.

\begin{Theorem} \label{hopefully...} \
Let $A$ be a nonnoetherian dimer algebra, and let $\psi: A \to A'$ a cyclic contraction.
\begin{enumerate}
 \item The reduced center $\hat{Z}$ and ghor center $R$ of $A$ are both depicted by the center $Z' \cong S$ of $A'$.
 \item The affine scheme $\operatorname{Spec}R$ and the reduced scheme of $\operatorname{Spec}Z$ are birational to the noetherian scheme $\operatorname{Spec}S$, and each contain precisely one closed point of positive geometric dimension, namely $\mathfrak{m}_0$ and $\mathfrak{z}_0$.
\end{enumerate}
\end{Theorem}

\begin{proof}
(1) We first claim that $\hat{Z}$ and $R$ are depicted by $S$.
By Theorem \ref{isolated sing},
$$U^*_{S/\hat{Z}} = U_{S/\hat{Z}} \not = \emptyset \ \ \text{ and } \ \ U^*_{S/R} = U_{S/R} \not = \emptyset.$$
Furthermore, by Lemma \ref{max ideal}, the morphisms $\iota_{S/\hat{Z}}$ and $\iota_{S/R}$ are surjective.\footnote{The fact that $S$ is a depiction of $R$ also follows from \cite[Theorem D.1]{B5}, since the algebra homomorphism $\tau_{\psi}: \Lambda \to M_{|Q_0|}(B)$ is an impression \cite[Theorem 5.9.1]{B2}.}

(2.i) As schemes, $\operatorname{Spec}S$ is isomorphic to $\operatorname{Spec}\hat{Z}$ and $\operatorname{Spec}R$ on the open dense subset $U_{S/\hat{Z}} = U_{S/R}$, by Theorem \ref{isolated sing}.
Thus all three schemes are birationally equivalent.  
Furthermore, $\operatorname{Spec}\hat{Z}$ and $\operatorname{Spec}R$ each contain precisely one closed point where $\hat{Z}$ and $R$ are locally nonnoetherian, namely $\mathfrak{z}_0$ and $\mathfrak{m}_0$, again by Theorem \ref{isolated sing}.

(2.ii.a) We claim that the closed point $\mathfrak{m}_0 \in \operatorname{Spec}R$ has positive geometric dimension.

Indeed, since $A$ is nonnoetherian, there is a cycle $p$ such that $\sigma \nmid \overbar{p}$ and $\overbar{p}^n \in S \setminus R$ for each $n \geq 1$, by Lemma \ref{cyclelemma2}.6. 

If $\overbar{p}$ is in $\mathfrak{m}_0S$, then there are monomials $g \in R$, $h \in S$, such that $gh = \overbar{p}$.
Furthermore, $g \not = \sigma^n$ for all $n \geq 1$ since $\sigma \nmid \overbar{p}$.
But then $\overbar{p} = gh$ is in $R$ by Lemma \ref{cyclelemma2}.1, a contradiction.
Therefore
$$\overbar{p} \not \in \mathfrak{m}_0S.$$

Consequently, for each $c \in k$, there is a maximal ideal $\mathfrak{n}_c \in \operatorname{Max}S$ such that
$$(\overbar{p} -c, \mathfrak{m}_0)S \subseteq \mathfrak{n}_c.$$
Thus,
$$\mathfrak{m}_0 \subseteq (\overbar{p} -c, \mathfrak{m}_0)S \cap R \subseteq \mathfrak{n}_c \cap R.$$
Whence $\mathfrak{n}_c \cap R = \mathfrak{m}_0$ since $\mathfrak{m}_0$ is maximal.
Therefore by Theorem \ref{isolated sing},
$$\mathfrak{n}_c \in U^c_{S/R}.$$

Set
\begin{equation} \label{yet}
\mathfrak{q} := \bigcap_{c \in k} \mathfrak{n}_c.
\end{equation}
The intersection of radical ideals is radical, and so $\mathfrak{q}$ is a radical ideal.
Thus, since $S$ is noetherian, the Lasker-Noether theorem implies that there are minimal primes $\mathfrak{q}_1, \ldots, \mathfrak{q}_{\ell} \in \operatorname{Spec}S$ over $\mathfrak{q}$ such that
\begin{equation} \label{not}
\mathfrak{q} = \mathfrak{q}_1 \cap \cdots \cap \mathfrak{q}_{\ell}.
\end{equation}

We claim that at least one $\mathfrak{q}_i$ is non-maximal in $S$.
Setting $\mathcal{Z}(I) := \mathcal{Z}_S(I)$, we have
$$\cup_{i = 1}^{\ell}\mathcal{Z}(\mathfrak{q}_i) = \mathcal{Z}(\cap_{i =1}^{\ell} \mathfrak{q}_i) \stackrel{\textsc{(i)}}{=} \mathcal{Z}(\mathfrak{q}) \stackrel{\textsc{(ii)}}{=} \mathcal{Z}(\cap_{c \in k} \mathfrak{n}_c) = \cup_{c \in k} \mathcal{Z}(\mathfrak{n}_c),$$
where (\textsc{i}) holds by (\ref{not}), and (\textsc{ii}) holds by (\ref{yet}).
Thus, if each $\mathfrak{q}_i$ were maximal, then 
$$\{ \mathcal{Z}(\mathfrak{q}_1), \ldots, \mathcal{Z}(\mathfrak{q}_{\ell}) \} = \{ \mathcal{Z}(\mathfrak{n}_c) : c \in k \}.$$
In particular, there there would be an infinite set of distinct zero-dimensional points that is equal to a finite set of zero-dimensional points, which is not possible.

Therefore at least one $\mathfrak{q}_i$ is a non-maximal prime, say $\mathfrak{q}_1$.
Whence,
$$\mathfrak{m}_0 = \bigcap_{c \in k} (\mathfrak{n}_c \cap R) = \bigcap_{c \in k} \mathfrak{n}_c \cap R = \mathfrak{q} \cap R \subseteq \mathfrak{q}_1 \cap R.$$
Consequently, $\mathfrak{q}_1 \cap R = \mathfrak{m}_0$ since $\mathfrak{m}_0$ is maximal in $R$.

Since $\mathfrak{q}_1$ is a non-maximal prime ideal of $S$,
$$\operatorname{ht}(\mathfrak{q}_1) < \operatorname{dim}S.$$
Furthermore, $S$ is a depiction of $R$ by Claim (1).
Thus
$$\operatorname{ght}(\mathfrak{m}_0) \leq \operatorname{ht}(\mathfrak{q}_1) < \operatorname{dim}S \stackrel{\textsc{(i)}}{=} \operatorname{dim}R,$$
where (\textsc{i}) holds by Theorem \ref{generically noetherian}.
Therefore
$$\operatorname{gdim} \mathfrak{m}_0 = \operatorname{dim}R - \operatorname{ght}(\mathfrak{m}_0) \geq 1,$$
proving our claim.

(2.ii.b) Finally, we claim that the closed point $\mathfrak{z}_0 \in \operatorname{Spec}\hat{Z}$ has positive geometric dimension.

Again, since $A$ is nonnoetherian, there is a cycle $p$ such that $\sigma \nmid \overbar{p}$ and $\overbar{p} \in S \setminus R$, by Lemma \ref{cyclelemma2}.6. 

If $\overbar{p}$ is in $\mathfrak{z}_0S$, then there are monomials $g \in \hat{Z}$, $h \in S$, such that $gh = \overbar{p}$.
Furthermore, $g \in R$ since $\hat{Z} \subseteq R$ by (\ref{ZKA}); and $g \not = \sigma^n$ for all $n \geq 1$ since $\sigma \nmid \overbar{p}$.
But then $\overbar{p} = gh$ is in $R$ by Lemma \ref{cyclelemma2}.1, a contradiction. 
Therefore
$$\overbar{p} \not \in \mathfrak{z}_0S.$$

The proof then follows as in Claim (2.ii.a).
\end{proof}

\begin{Remark} \rm{
Although $\hat{Z}$ and $R$ determine the same variety using depictions, their associated affine schemes
$$(\operatorname{Spec}\hat{Z}, \mathcal{O}_{\hat{Z}}) \ \ \ \text{ and } \ \ \ \left( \operatorname{Spec}R, \mathcal{O}_{R} \right)$$
will not be isomorphic if their rings of global sections, $\hat{Z}$ and $R$, are not isomorphic.
}\end{Remark}

\ \\
\textbf{Acknowledgments.}
The author would like to thank an anonymous referee for their helpful comments which have improved the article.
The author was supported by the Austrian Science Fund (FWF) grant P 30549-N26.
Part of this article is based on work supported by the Heilbronn Institute for Mathematical Research.

\bibliographystyle{hep}
\def\cprime{$'$} \def\cprime{$'$}

\end{document}